\documentclass[10pt,a4paper,twocolumn]{article}

\usepackage{amsmath,amssymb,amsfonts,amsthm}
\usepackage{algorithmic}
\usepackage{graphicx}
\usepackage{xcolor}

\usepackage{amsfonts} %Always load this after mathalfa
\usepackage{amssymb} %Always load this after mathalfa

\usepackage{graphicx}
\usepackage{mathtools}
\usepackage{cleveref}
\usepackage[top=1cm,left=2cm,right=2cm,bottom=2cm]{geometry}
\usepackage{todonotes}

\usepackage[bibencoding=auto,backend=biber,sorting=none,giveninits=true]{biblatex}
\newtheorem{theorem}{Theorem}
\newtheorem{lemma}[theorem]{Lemma}

\newtheorem{defn}[theorem]{Definition}

\newtheorem{remark}[theorem]{Remark}

\newcommand\mynewcommand[1]{\let#1\relax\newcommand#1}
\mynewcommand{\f}[1]{\mathfrak{#1}}
\mynewcommand{\c}[1]{\mathcal{#1}}
\mynewcommand{\b}[1]{\mathbb{#1}}
\mynewcommand{\bf}[1]{\bm{#1}}
\mynewcommand{\bs}[1]{\boldsymbol{#1}}
\mynewcommand{\leq}{\leqslant}
\mynewcommand{\geq}{\geqslant}
\mynewcommand{\Re}{\operatorname{Re}}
\mynewcommand{\Im}{\operatorname{Im}}
\DeclarePairedDelimiter{\abs}{\lvert}{\rvert}
\mynewcommand{\<}{\left\langle}
\mynewcommand{\>}{\right\rangle}

\DeclarePairedDelimiter{\norm}{\lVert}{\rVert}
\mynewcommand{\intg}[3]{\int_{#1} #2 \,\text{d}#3}
\mynewcommand{\eps}{\varepsilon}
\mynewcommand{\phi}{\varphi}
\DeclareMathOperator{\tr}{tr}
\renewcommand{\mod}{\operatorname{mod}}

\newcommand{\mat}[1]{\left(\begin{matrix}#1\end{matrix}\right)}

\renewcommand{\mod}{\operatorname{mod}}

\usepackage{tikz-cd}
\tikzcdset{every label/.append style = {font = \small}}

\usepackage{tikz}
\usetikzlibrary{shapes,arrows,positioning,calc,external}
\tikzset{
	block/.style = {draw, fill=white, rectangle, minimum height=3em, minimum width=3em},
	tmp/.style  = {coordinate},
	sum/.style= {draw, fill=white, circle, node distance=1cm, inner sep=0pt},
	input/.style = {coordinate},
	output/.style= {coordinate},
	pinstyle/.style = {pin edge={to-,thin,black}
	}
}

\usepackage{sectsty}
\sectionfont{\fontsize{12}{12}\selectfont}
\subsectionfont{\fontsize{11}{11}\selectfont}

\begin{document}
	\title{On the exactness of a stability test for Lur'e systems with slope-restricted nonlinearities}
	\author{Andrey Kharitenko, Carsten W. Scherer}
	
	\maketitle
	
	\begin{abstract}
		In this note it is shown that the famous multiplier absolute stability test of R. O'Shea, G. Zames and P. Falb is necessary and sufficient if the set of Lur'e interconnections is lifted to a Kronecker structure and an explicit method to construct the destabilizing static nonlinearity is presented.
	\end{abstract}
	\renewcommand{\thesection}{}
	
	\renewcommand{\thesection}{\arabic{section}}
	\setcounter{section}{0}
	
	\section{Introduction}
	\label{sec:introduction}
	A classical problem in control theory is the stability analysis of the so-called Lur'e systems.
	Studied by A. Lur'e and V. Postnikov in 1944 \cite{lure1944on} explicitly for the first time, they consist of a feedback interconnection between a linear time-invariant system $G$ and a nonlinear static operator~$\Delta$.
	
	The method of stability multipliers \cite{desoer2009feedback, willems1971analysis} is used to establish stability of such systems by searching for an artificial system $M \in \c{M}_{\Delta}$ (called multiplier) such that $-MG$ is strictly passive. Here
	$\c{M}_{\Delta}$ is a suitable set of multipliers such that $\Delta M^{-1}$ is a positive operator for any $M \in \c{M}_{\Delta}$.
	Under suitable assumptions, if such a multiplier is found, stability of the interconnection can be deduced from the passivity- or IQC-theorem
	\cite{desoer2009feedback}, \cite{megretski1997system}.
	Different classes of nonlinearities $\bs{\Delta}$ allow for different multiplier classes $ \mathcal{M}_{\bs{\Delta}}$, and a larger multiplier class implies a less conservative stability test.
	
	Yet the question of whether a stability criterion created in this way is necessary is dependent on $\mathcal{M}_{\bs{\Delta}}$ and remains open in general.
	Many well-known multiplier stability criteria, such as the Popov or circle criterion, were shown already early on to be only sufficient \cite{pyatnitskii1973existence, brockett1965optimization, brockett1966status}.
	For the class of monotone or slope-restricted nonlinearities, a particularly rich class of multipliers is given by the so-called O'Shea-Zames-Falb (OZF) multipliers \cite{zames1968stability,willems1968some,carrasco2016zames}, which were introduced in \cite{zames1968stability} and \cite{willems1968some} for continuous- and discrete-time SISO interconnections and extended in \cite{safonov2000zames} to the MIMO case for monotone nonlinearities.
	In fact, for this class of nonlinearities in discrete-time, OZF-multipliers form the largest possible passivity multiplier set \cite{mancera2005all} and give the least conservative results.
	Along with the discovery of suitable search methods \cite{chen1995robustness,chang2011computation,carrasco2016zames,carrasco2019convex,fetzer2017absolute}, this has motivated its extensive use in recent times, e.g. \cite{freeman2018noncausal,michalowsky2021robust,fetzer2017zames}.
	
	The question about the conservatism of this multiplier set is, therefore, of great interest \cite{zhang2021duality, khong2021necessity, su2021necessity} and its exactness was conjectured in \cite{carrasco2016zames}.
	Moreover, this question is closely related to the Kalman problem from the 1950s, which asks for necessary and sufficient conditions for the stability of Lur'e interconnections with slope-restricted nonlinearities.
	
	In this note we show that the stability test generated by OZF-multipliers cannot differentiate between Lur'e interconnections in different dimensions, if the linear part is lifted from $G$ to the Kronecker form $G \otimes I_d$, and could, hence, be potentially conservative.
	As our main result, we prove that, up to this lifting, the test is indeed necessary and sufficient and provide an explicit construction of the destabilizing nonlinearity.
	We also point out connections to the duality bounds obtained in \cite{zhang2021duality} and to criteria for the absence of periodic oscillations in nonlinear filters \cite{claasen1975frequency}.
	
	The paper is structured as follows.
	After introducing the necessary concepts of absolute stability and integral quadratic constraints as well as stating the main stability test in Section \ref{sec:preliminary}, we proceed in Section \ref{sec:observation} to show that the test cannot differentiate between interconnections up to a lifting.
	In Section \ref{sec:main} we then state our main exactness result.
	To prove it, we reformulate the infeasibility of the main stability test in Section \ref{sec:linearProgram} as a condition on a linear program and explicitly construct a nonlinearity using duality in Section \ref{sec:constructionUncertainty}.
	Finally, in Section \ref{sec:proofMain} we prove our main result and discuss its connections to other results in the literature in Section \ref{sec:discussion}.
	All technical proofs, except the one of the main exactness result, can be found in Appendix~\ref{sec:proofs}, while Appendix \ref{sec:auxiliary} contains two auxiliary facts.
	
	\section{Notation and Preliminary Results}
	\label{sec:preliminary}
	
	\subsection{Notation}
	
	The standard inner product in $\mathbb{R}^d$ is denoted as $\<u,v\> = u^\top v$, while
	$\norm{M} = \bar{\sigma}(M)$ is the spectral norm of a real matrix.
	The identity matrix in $\b{R}^{d \times d}$ is denoted by $I_d$.
	Moreover,
	$\ell_d^{2e}$ denotes the linear space of all sequences $x:\mathbb{N}_0\to\mathbb{R}^d$, while
	$\ell_d^2$ is the subspace of square summable sequences equipped with the inner product
	$\<x,y\> = \sum_{k=0}^\infty \<x_k,y_k\>$.
	The power of a signal $x \in \ell_d^{2e}$ is defined by $\operatorname{pow}(x)^2 = \limsup_{N \to \infty} \frac{1}{N}\sum_{k=0}^N \norm{x_k}^2\in[0,\infty]$.
	The space of linear bounded operators on $\ell_d^2$
	and the (induced) norm thereon are denoted by $\mathcal{L}(\ell_d^2)$ and $\|.\|$%_{\mathcal{L}(\ell_2^d)}$
	, respectively.
	The Kronecker product of two matrices $A$ and $B$ is denoted by $A \otimes B$ and the direct sum of two vector spaces $X$ and $Y$ by $X \oplus Y$.
	%Also let $\ell_{\infty} := \ell_{\infty}^{1 \times 1}$ and then we use the more common notation $x_N = X^N$.
	If $N \in \b{N}$ and $x \in \ell_d^{2e}$, then $P_Nx$ denotes the cutoff projection
	defined as $(P_N x)_k = x_k$ if $k\leq N$ and $(P_Nx)_k = 0$ for $k > N$.
	A relation $R \subseteq \ell_m^{2e} \times \ell_l^{2e}$ is said to be bounded if there exist $\gamma$ and $\beta$ such that for all $(x,y) \in R$ and $N \in \b{N}_0$ we have $\norm{P_N y} \leq \gamma \norm{P_N x} + \beta$.
	The infimum of all such values $\gamma$, called the gain of $R$, is denoted by $\norm{R}$.
	The relation $R$ is said to be total if for each $x \in \ell_m^{2e}$ there exists some $y \in \ell_l^{2e}$ with $(x,y) \in R$ and causal if for all $N \in \b{N}_0$ and all $(x,y), (u,v) \in R$ such that $P_N x = P_N u$ it follows that there exists some $w \in \ell_{l}^{2e}$ with $(u,w) \in R$ and $P_N y = P_N w$.
	Operators $G:\ell_m^{2e} \to \ell_l^{2e}$ are identified with their graph relation $R = \{(x,G(x)) \mid x \in \ell_m^{2e}\}$.
	The transfer function of a stable finite-dimensional linear time-invariant (LTI) system $G$ is denoted by $G(z)$.
	Finally we define the unit circle by $\b{T} = \{z \in \b{C} \mid \abs{z} = 1\}$.

	\subsection{Well-posedness and stability}
	%
	%We consider the standard performance feedback interconnection:
	%\begin{align}
	%	\label{eq:performanceFeedbackInterconnection}
	%	\mat{e \\ z}
	%	= G \mat{d \\ w}
	%	= \mat{G_{ed} & G_{ew} \\ G_{zd} & G_{zw}} \mat{d \\ w}
	%	\text{\   \ and \   \ }
	%	(z+v,w-u) \in \Delta
	%\end{align}
	%
	%where $G:\ell_{m+m_p}^{2e} \to \ell_{l+l_p}^{2e}$ is linear, causal and bounded and where $\Delta \subseteq \ell_{l}^{2e} \times \ell_{m}^{2e}$ is a relation.
	%We also denote by $R(\Delta) \subseteq \ell_{m_p+l+m}^{2e} \times \ell_{l_p+l+m}^{2e}$ the relation between all inputs $(d,u,v)$ and outputs $(e,w,z)$ defined by (\ref{eq:performanceFeedbackInterconnection}).
	%We say that the feedback interconnection (\ref{eq:performanceFeedbackInterconnection}) is well-posed, if $R(\Delta)$ is total and causal, (internally) stable, if $R(\Delta)$ is bounded.
	%Additionally we define the channel $d \mapsto e$ relation $R^{de}(\Delta) = \{(d,e) \mid \exists w,z:\, \text{(\ref{eq:performanceFeedbackInterconnection}) holds with } u = v = 0\}$ and say that the interconnection achieves a performance level of $\gamma > 0$ if $\norm{R^{de}(\Delta)} \leq \gamma$.
	
	Let $G:\ell_{m}^{2e} \to \ell_{l}^{2e}$ be linear and bounded and $\Delta \subseteq \ell_{l}^{2e} \times \ell_{m}^{2e}$ a relation.
	Consider the standard feedback interconnection (Fig. \ref{fig:feedback}) between $G$ and $\Delta$ defined by
	\begin{align}
		\label{eq:iqcRelationInterconnection}
		e_2 = Ge_1 + u_2 \text{\ and\ } (e_2, e_1-u_1) \in \Delta,
	\end{align}
	where the input $u = (u_1,u_2)$ belongs to $\ell_{m}^{2e} \times \ell_{l}^{2e} = \ell_{m+l}^{2e}$.
	The interconnection (\ref{eq:iqcRelationInterconnection}) is said to be well-posed if the interconnection relation
	\begin{align}
		\label{eq:interconnectionRelation}
		R(\Delta) = \{(u,e) \in \ell_{l+m}^{2e} \times \ell_{l+m}^{2e} \mid \eqref{eq:iqcRelationInterconnection} \text{\ is satisfied}\}
	\end{align}
	with $e=(e_1,e_2)$ is total and causal.
	It is stable if, additionally, $R(\Delta)$ is bounded.
	%Note that we do not require the response to be unique, since we want to consider true relations $\Delta$.
	If $\bs{\Delta}$ is a set of relations, we say that (\ref{eq:iqcRelationInterconnection}) is robustly stable against $\bs{\Delta}$ if it is stable for each $\Delta \in \bs{\Delta}$ and $\sup_{\Delta \in \bs{\Delta}} \norm{R(\Delta)} < \infty$ holds.
	Note that this is a uniform notion of stability for the class $\bs{\Delta}$.
	
	\begin{figure}[h]%
		\centering
		\includegraphics[scale=1]{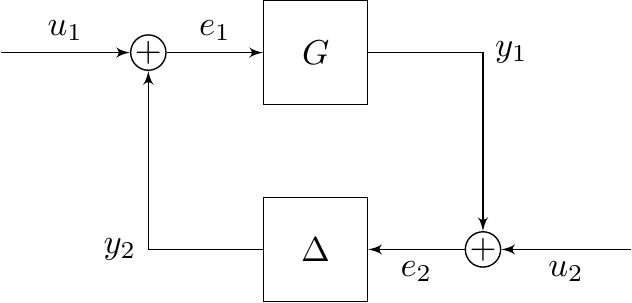}
		%	\subfloat[Modified feedback interconnection]{
		%		\includegraphics[scale=0.8]{mod_feedback.pdf} \label{fig:multFeedback}}
		\caption{Standard feedback interconnection}
		\label{fig:feedback}
	\end{figure}
	
	\subsection{Slope-restricted relations}
	
	In the following we consider any relation $f \subseteq \mathbb{R}^d \times \mathbb{R}^d$ as a multifunction with domain $\operatorname{dom}(f) = \{x \in \mathbb{R}^d \mid \exists y \in \mathbb{R}^d: (x,y) \in f\}$ and the value $f(x) = \{y \in \mathbb{R}^d \mid (x,y) \in f\}$ at $x \in \b{R}^d$; we write $f:\mathbb{R}^d \rightrightarrows \mathbb{R}^d$.
	
	A multifunction $f:\mathbb{R}^d \rightrightarrows \mathbb{R}^d$ is said to be cyclically monotone if for any cyclic sequence $x_1,\ldots,x_n \in \operatorname{dom}(f)$, $x_{n+1} = x_1$, and $y_j \in f(x_j)$, $j=1,\ldots,n$ it holds that $\sum_{j=1}^n \<y_j,x_{j+1}-x_j\> \leq 0$.
	%The multifunction $f$ is said to be maximally (cyclically) monotone, if there is no $g \supsetneq f$ with $g$ (cyclically) monotone.
	%By a theorem of T. Rockafellar \cite[Theorem 12.25]{rockafellar2009variational} every cyclically multifunction $f$ has a potential: $f \subseteq \partial F$ for some convex, lower-semicontinuous (lsc) and proper $F:\b{R}^d \to \b{R} \cup \{\infty\}$.
	
	Following \cite{freeman2018noncausal} and with $\lambda\in(-\infty,0]$, $\kappa\in(0,\infty]$, we define for every $f:\mathbb{R}^d \rightrightarrows \mathbb{R}^d$ the multifunction $T_{\lambda,\kappa,f} = (f-\lambda I) \circ (I-\frac{1}{\kappa}f)^{-1}$ so that $y \in T_{\lambda,\kappa,f}(x)$ iff there exists some $z \in \operatorname{dom}(f)$ and $u \in f(z)$ with $x=z-\frac{1}{\kappa}u$ and $y = u-\lambda z$.

	Here, we use the convention $1/\infty=0$ and note
	that $T_{0,\infty,f} = f$ by definition.
	The multifunction $f:\b{R}^d \rightrightarrows \b{R}^d$ is said to be $[\lambda,\kappa]$-slope-restricted (written $f \in \operatorname{slope}_d[\lambda,\kappa]$) if $T_{\lambda,\kappa,f}$ is cyclically monotone.
	
	Every relation $f \subseteq \b{R}^d \times \b{R}^d$ defines a static relation $\Delta = \Delta_f \subseteq \ell_d^{2e} \times \ell_d^{2e}$ on signals by $(w,z) \in \Delta$ iff $z_k \in f(w_k)$ for all $k \in \b{N}_0$.
	In the sequel  we work with the class
	\begin{align*}
		\bs{\Delta}_d(\lambda,\kappa) := \{\Delta_f \mid f \in \operatorname{slope}_d[\lambda,\kappa] \text{\ total and\ } 0 \in f(0)\}\,.
	\end{align*}

	\subsection{Stability multipliers}
	
	Any shift-invariant operator $M \in \mathcal{L}(\ell_d^2)$ has an infinite block-Toeplitz representation $M = (M_{i-j})_{i,j=0}^\infty$ with $M_{j} \in \b{R}^{d \times d}$.
	Conversely, any block-Toeplitz matrix $(M_{i-j})_{i,j=0}^\infty$ that satisfies $\sum_{j=-\infty}^\infty \norm{M_j} < \infty$ %, i.e.
	%\begin{align}
	%	\label{eq:finiteNorm}
	%	\sum_{j=-\infty}^\infty \norm{M_{j}} \leq C
	%\end{align}
	%for some $C > 0$,
	defines an operator $M \in \mathcal{L}(\ell_d^2)$ by $Mx = (\sum_{j=0}^\infty M_{ij} x_j)_{i=0}^\infty$. %with $\norm{M} \leq C$.
	A scalar matrix $M = (m_{i-j})_{i,j=0}^\infty \in \mathbb{R}^{\mathbb{N}_0 \times \mathbb{N}_0}$ is said to be doubly hyperdominant, if $m_{j} \leq 0$ for $j \neq 0$ and $\sum_{j=-\infty}^\infty m_{j} \geq 0$
	where the convergence of the series is assumed by definition.
	We denote the set of all doubly hyperdominant Toeplitz matrices by $\c{H}$.
	The above terminology is also used for finite matrices $M \in \b{R}^{N \times N}$ whenever $N$ is clear from the context.
	A block-Toeplitz matrix $M = (M_{i-j})_{i,j=0}^\infty$ is said to be doubly hyperdominant, if it is doubly-hyperdominant as a scalar matrix.
	%i.e., $(M_{j})_{kl} \leq 0$ for $j \neq 0$ and all $k,l \in \{1,\ldots,d\}$, $(M_{0})_{kl} \leq 0$ for $k,l \in \{1,\ldots,d\}$ with $k \neq l$ and $\sum_{j=-\infty}^\infty M_j$ is doubly hyperdominant.
	%The next theorem is a slight generalization of analogous results in \cite{mancera2005all}, \cite{michalowsky2020robust}:
	%\begin{theorem}
	%	\label{thm:positivityMultiplier}
	%	Let $\Delta \in \bs{\Delta}_d(\lambda,\kappa)$ and suppose that $M \in \c{H} \otimes I_d$ is a bounded operator.
	%	Then for any $x \in \ell_2^d$ and $y \in \ell_2^d$ with $(x,y) \in \Delta$ we have
	%	\begin{align*}
	%		\<\mat{x \\ y}, \Pi_M(\lambda,\kappa) \mat{x \\ y}\> \geq 0\,.
	%	\end{align*}
	%\end{theorem}
	
	The IQC theorem \cite{scherer2000linear, scherer2018iqc, veenman2016robust,kharitenko2022some} permits to verify the stability of the interconnection of $G$ and $\Delta_f$ as follows.
	
	\begin{theorem}
		\label{thm:stabilityIQC}
		Let $G$ be linear, bounded and causal on $\ell_d^2$. % and let $\kappa > 0$ and $\lambda \leq 0$.
		Assume that the feedback interconnection of $G$ and $\Delta$ is well-posed for every $\Delta \in \bs{\Delta}_d(\lambda,\kappa)$.
		Then the interconnection is robustly stable against $\bs{\Delta}_d(\lambda,\kappa)$ if there exists some $\epsilon > 0$ and $M \in \c{H} \otimes I_d$ such that
		\begin{align}
			\<\mat{Gw \\ w},\Pi_M(\lambda,\kappa) \mat{Gw \\ w}\>\leq -\epsilon \norm{w}_{2}^2
			\text{\ for all\ } w \in \ell_d^2\,,
			\label{eq:negativeMultiplier}
		\end{align}
		where
		\begin{align*}
			\Pi_M(\lambda,\kappa) := \mat{-\lambda (M+M^*) & M^*+\frac{\lambda}{\kappa}M \\ M+\frac{\lambda}{\kappa}M^* & - \frac{1}{\kappa}(M + M^*)}.
		\end{align*}
	\end{theorem}
	\mbox{}\\
	This is the Zames-Falb stability test that is the subject of study in this note.
	
	%\begin{theorem}
	%	\label{thm:stabilityIQC}
	%	Let $G:\ell_{m_p+m}^{2e} \to \ell_{l_p+l}^{2e}$ be linear, bounded and causal.
	%	Assume that the interconnection (\ref{eq:performanceFeedbackInterconnection}) of $G$ and $\Delta_f$ is well-posed for every total cyclically monotone $f \in \b{R}$ (resp. $f$ monotone).
	%	Then the interconnection (\ref{eq:performanceFeedbackInterconnection}) achieves performance level $\gamma$ if there exists some $\epsilon > 0$ and doubly hyperdominant (resp. doubly hyperdominant, symmetric) bounded LTI $M = (m_{i-j})_{i,j=0}^\infty \otimes I_d = (m_{i-j}I_d)_{i,j=0}^\infty$ with
	%	\begin{align}
	%		\mat{G_{ed} & G_{ew}\\ I & 0}^* \mat{0 & M^* \\ M & 0} \mat{G_{ed} & G_{ew}\\ I & 0} + \mat{G_{zd} & G_{zw}\\ 0 & I}^* \mat{\frac{1}{\gamma}I & 0 \\ 0 & -\gamma I} \mat{G_{zd} & G_{zw}\\ 0 & I}\leq -\epsilon I
	%		\label{eq:negativePerformanceMultiplier}
	%	\end{align}
	%\end{theorem}
	
	\subsection{Positive definite functions and sequences}
	\label{sec:positiveDefiniteFunctions}
	A function $f:\b{T} \to \b{C}$ is said to be positive definite (p.d.) if $f(z^{-1}) = f(z)^*$ for all $z \in \b{T}$ and for any finite subset $\{z_1,\ldots,z_n\} \subseteq \b{T}$ the matrix $(f(z_kz_j^{-1}))_{k,j=1}^n$ is hermitian and positive semi-definite.
	Any p.d. function $f$ satisfies $f(1) \geq 0$ as well as $f(z)^* = f(z^*)$ and $\abs{f(z)} \leq f(1)$ for any $z \in \b{T}$.
	A well-known theorem from harmonic analysis reads as follows.
	\begin{theorem}[Bochner]
		\label{thm:bochner}
		For a function $f:\b{T} \to \b{C}$ the following statements are equivalent:
		\begin{enumerate}
			\item $f$ is continuous and positive definite
			
			\item there exists some finite, nonnegative measure $\mu$ on $\b{Z}$ such that $f(z) = \intg{\b{Z}}{z^{t}}{\mu(t)}$ for all $z \in \b{T}$.
		\end{enumerate}
	\end{theorem}
	A finite set $\{c_j\}_{j=0}^{N-1}$ of complex numbers is said to be positive definite if the circulant matrix $C = (c_{j-k \mod N})_{j,k=0}^{N-1}$ is hermitian and positive semi-definite.
	If $f:\b{T} \to \b{C}$ is p.d. and $\{z_0,\ldots,z_{N-1}\} \subseteq \b{T}$ is the set of the $N$-th roots of unity, then the sequence defined by $c_j = f(z_j)$ for $j=0,\ldots,N-1$ is p.d..
	Conversely, if we are given a p.d. sequence $\{c_j\}_{j=0}^{N-1}$, one possible p.d. interpolant is given by the next theorem, which is a discrete analogue of the one proven in \cite{belov2013positive}.
	
	\begin{theorem}
		\label{thm:belov1}
		Let $N \in \b{N}$ and let $\{z_0,\ldots,z_{N-1}\}$ be the collection of $N$-th roots of unity.
		Then the piecewise-linear function $f$ with interpolation nodes $\{z_0,\ldots,z_{N-1}\}$ satisfying $f(z_j) = c_j \in \b{C}$ for $j \in \{0,\ldots,N-1\}$ is positive definite iff $\{c_j\}_{j=0}^{N-1}$ is positive definite.
	\end{theorem}
	
	\begin{proof}
		The proof is similar to the one of \cite[Theorem 2]{belov2013positive} and can be found in \cite{kharitenko2022some}.
	\end{proof}
	Here we call a function $f:\b{T} \to \b{C}$ piecewise-linear with interpolation nodes $\{e^{i\omega_0},\ldots,e^{i\omega_{N-1}}\} \subseteq \b{T}$ for
	$0 \leq \omega_0 < \ldots < \omega_{N-1} < 2\pi$ and $\omega_{N} = \omega_0$, if
	\begin{align*}
		f(e^{i((1-t)\omega_j + t\omega_{j+1})}) = (1-t)f(e^{i\omega_j}) + t f(e^{i\omega_{j+1}})
	\end{align*}
	holds for all $t \in [0,1]$ and $j=0,\ldots,N-1$.

	\section{Exactness results}
	\label{sec:exactness}
	Now we come to the analysis of the criterion given by Theorem \ref{thm:stabilityIQC}.
	
	\subsection{An observation about Theorem \ref{thm:stabilityIQC}}
	\label{sec:observation}
	Assume that $G:\ell^{2e} \to \ell^{2e}$ and that \eqref{eq:negativeMultiplier} is satisfied for some $\epsilon > 0$ and $M \in \c{H}$, where $d = 1$.
	If we denote
	\begin{align*}
		T = \mat{G \\ I}^* \Pi_M(\lambda,\kappa) \mat{G \\ I} \in \c{L}(\ell^2)
	\end{align*}
	and stack $d$ copies of \eqref{eq:negativeMultiplier}, this condition is equivalent to
	\begin{align}
		\label{eq:kroneckerFDI}
		\< (I_d \otimes T) w, w\> \leq -\epsilon \norm{w}^2 \text{\ for all\ } w= \mat{w^{(1)} \\ \vdots \\ w^{(d)}} \in \bigoplus_{k=1}^d \ell^2\,,
	\end{align}
	where $I_d \otimes T = \textstyle\bigoplus_{k=1}^d T \in \c{L}(\bigoplus_{k=1}^d \ell^2)$.
	By identifying $\bigoplus_{k=1}^d \ell^2$ with $\ell_d^2$ it is easy to see that \eqref{eq:kroneckerFDI} is equivalent to
	\begin{align}
		\< (T \otimes I_d) w, w\> \leq -\epsilon \norm{w}^2 \text{\ for all\ } w \in \ell_d^2\,.
	\end{align}
	%If $R:\ell_d^2 \to \bigoplus_{k=1}^d \ell^2$ is the unitary operator with $Rw = \oplus_{l=1}^d(w_k^{(l)})_{k=0}^\infty$ for $w = (\oplus_{l=1}^d w_k^{(l)})_{k=0}^\infty$, then
	%(\ref{eq:kroneckerFDI}) is clearly equivalent to
	%\begin{align*}
	%	\<R^* (I_d \otimes T) R w, w\>
	%	&= \<(I_d \otimes T) R w,R w\> \\
	%	&\leq -\epsilon \norm{Rw}^2 = -\epsilon \norm{w}^2 \text{\ for all\ } w \in \ell_d^2\,.
	%\end{align*}
	%Now, it is easy to see that in terms of (infinite) matrices we have $R^*(I_d \otimes T)R = T \otimes I_d$ and we write $T \otimes I_d$ also for the corresponding operator in $\c{L}(\ell_d^2)$.
	Then, again by a simple computation, we note that
	\begin{align*}
		T \otimes I_d = \mat{G \otimes I_d \\ I}^* \Pi_{M \otimes I_d}(\lambda,\kappa) \mat{G \otimes I_d \\ I}\,.
	\end{align*}

	Thus \eqref{eq:negativeMultiplier} holds for the system $G \otimes I_d \in \c{L}(\ell_d^2)$, the multiplier $\bar{M} = M \otimes I_d \in \c{H} \otimes I_d$ and the same $\epsilon > 0$.
	Conversely, if \eqref{eq:negativeMultiplier} holds for $G \otimes I_d \in \c{L}(\ell_d^2)$ and with the multiplier $M \otimes I_d$, then \eqref{eq:negativeMultiplier} is also satisfied for $G$ with the multiplier $M$.
	
	This permits us to draw the following conclusion:
	%Hence we have the following observation:
	\emph{The stability test of Theorem \ref{thm:stabilityIQC} cannot differentiate between the robust stability of the interconnection between $G$ and $\bs{\Delta}_1(\lambda,\kappa)$ and robust stability of the interconnection between $G \otimes I_d$ and $\bs{\Delta}_d(\lambda,\kappa)$ for any $d \in \b{N}_0$.}
	
	\subsection{Main theorem}
	\label{sec:main}
	In view of the latter observation and the fact that \eqref{eq:negativeMultiplier} is homogeneous in $\epsilon$ and $M$, we focus on multipliers that are normalized as $M \in \c{H}_1 := \{M = (m_{i-j})_{i,j=0}^{\infty}\in \c{H} \mid m_0 \leq 1\} \subseteq \c{H}$.
	For the ease of exposition, we also set $\lambda = 0$.
	The following theorem is main result of this note.
	
	\begin{theorem}
		\label{thm:mainNecessity}
		Let $G:\ell^2 \to \ell^2$ be LTI, bounded and causal with a finite-dimensional state-space representation. %, $\kappa > 0$ and $\lambda \leq 0$.
		If for $d=1$ and some $\epsilon > 0$ there exists no $M \in \c{H}_1$ such that \eqref{eq:negativeMultiplier} is satisfied, then there exists some $d \in \b{N}$ and $\Delta \in \bs{\Delta}_d(0,\kappa)$ such that the interconnection (\ref{eq:interconnectionRelation}) between $G \otimes I_d$ and $\Delta$ has a gain of at least $1/\sqrt{8\epsilon}$.
	\end{theorem}
	In short, the test provided by OZF-multipliers is exact up to a lifting of the original system $G$ to $G \otimes I_d$. Precisely,
	let $\epsilon_{\max}$ be the supremum of all $\epsilon\geq 0$ such that \eqref{eq:negativeMultiplier} is satisfied for some $M \in \c{H}_1$.
	Then the  Zames-Falb stability test is successful iff $\epsilon_{\max}>0$, and
	our main result implies
	$$\sup_{d\in\b{N},\,\Delta \in \bs{\Delta}_d(0,\kappa)} \norm{R(\Delta)} \geq\frac{1}{\sqrt{8\epsilon_{\max}}}.$$
	This not only shows robust instability of the loop in case of $\epsilon_{\max}=0$, but also
	reveals that $\epsilon_{\max}$ can be viewed as a stability margin in case it is positive.

	%Thus, the test provided by Zames-Falb multipliers is exact up to a lifting of the original system $G$ to $G \otimes I_d$ and the parameter $\epsilon$ in the linear operator inequality \eqref{eq:negativeMultiplier} acts as a stability margin \magc{in the following sense}: \red{If $\epsilon_{\max}$ is the maximal $\epsilon > 0$ such that \eqref{eq:negativeMultiplier} is satisfied for some $M \in \c{H}_1$, then for $\epsilon_{\max} \to 0$ the gain of the interconnection grows without bound.}
	
	\subsection{Linear program formulation}
	\label{sec:linearProgram}
	In order to understand what the infeasibility of \eqref{eq:negativeMultiplier} entails, we restate \eqref{eq:negativeMultiplier} as a frequency domain inequality (FDI).
	We first set $\kappa = \infty$.
	It is easy to see that a generic $M \in \c{H}_1$ can be written as $M = I - H$ with $H$ being Toeplitz and doubly substochastic, i.e. $H = (h_{j-i})_{i,j=0}^\infty$ for some nonnegative sequence $h_{k}$ such that $\sum_{k=-\infty}^\infty h_k \leq 1$.
	Taking the z-transform of \eqref{eq:negativeMultiplier} yields the FDI
	\begin{align}
		\label{eq:positiveDefiniteFDI}
		\Re(G(z)(1-H(z)) \leq -\frac{\epsilon}{2} \text{\ for all\ } z \in \b{T}\,.
	\end{align}
	Here $H(z) = \sum_{k=-\infty}^\infty h_k z^{-k}$ is continuous and p.d. by Theorem \ref{thm:bochner} and satisfies $H(1) \leq 1$.
	Hence the test in Theorem \ref{thm:stabilityIQC} is equivalent to finding some continuous, p.d. $H:\b{T} \to \b{C}$ with \eqref{eq:positiveDefiniteFDI}.
	
	As a next step we relate the FDI \eqref{eq:positiveDefiniteFDI} to a family of finite-dimensional convex programs.
	For this purpose, we evaluate the FDI on the $N$-th roots of unity
	$z_j = \exp(2\pi ij/N)\in \b{T}$, $j=0,\ldots,N-1$.
	%$\{z_0,\ldots,z_{N-1}\} \subseteq \b{T}$ \magc{where} $z_j = \exp(2\pi ij/N)$.
	
	On the one hand, if there exists a function $H:\b{T} \to \b{C}$ with $H(1) \leq 1$ which satisfies \eqref{eq:positiveDefiniteFDI}, we infer from Section \ref{sec:positiveDefiniteFunctions} that,
	for every $N\in\b{N}$, the convex program
	\begin{align}
		\label{eq:pdProgram}
		\Re(G(z_j)(1-c_j)) \leq -\delta \text{\ for all\ } j=0,\ldots,N-1
	\end{align}
	with $\delta = \epsilon/2$ is feasible in the convex set of p.d. sequences $(c_{j})_{j=0}^{N-1}$ with $c_0 \leq 1$.
	Note that
	this observation is closely related to the duality results obtained in \cite{zhang2021duality}, which will be discussed further below.
	
	It has the following converse.
	%From Section \ref{sec:positiveDefiniteFunctions} it follows now that there cannot exist any continuous p.d. function $H:\b{T} \to \b{C}$ with $H(1) \leq 1$ satisfying \eqref{eq:positiveDefiniteFDI} if the convex program
	%\begin{align}
	%	\label{eq:pdProgram}
	%	\Re(G(z_j)(1-c_j)) \leq -\delta \text{\ for all\ } j=0,\ldots,N-1
	%\end{align}
	%with $\delta = \epsilon/2$ is infeasible in the convex set of p.d. sequences $(c_{j})_{j=0}^{N-1}$ with $c_0 \leq 1$.
	%This observation is closely related to the duality results obtained in \cite{zhang2021duality}, which will be discussed further below, and has the following converse.
	
	\begin{lemma}
		\label{thm:nonexistenceImpliesInfeasible}
		Let $\epsilon > 0$ and assume that there is no $M \in \c{H}_1$ such that \eqref{eq:negativeMultiplier} is satisfied.
		Then for each $\delta > \epsilon/2$ there is some $N = N_\delta \in \b{N}_0$ such that (\ref{eq:pdProgram}) is infeasible.
	\end{lemma}
	
	The proof relies on multipliers with transfer functions $M(z) = H(1) - H(z)$ for some piecewise-linear positive definite $H:\b{T} \to \b{C}$.
	These are called piecewise-linear OZF-multipliers.
	Since they have non-rational transfer functions, they do not have a finite-dimensional state-space representation and are, hence, different from the popular so-called finite impulse response (FIR) multipliers \cite{carrasco2019convex,fetzer2017absolute}.\\
	Fig. \ref{fig:pwlNyquist} shows the Nyquist plot of a piecewise-linear Zames-Falb multiplier.
	\begin{figure*}[h]
		\centering
		\includegraphics[scale=0.7]{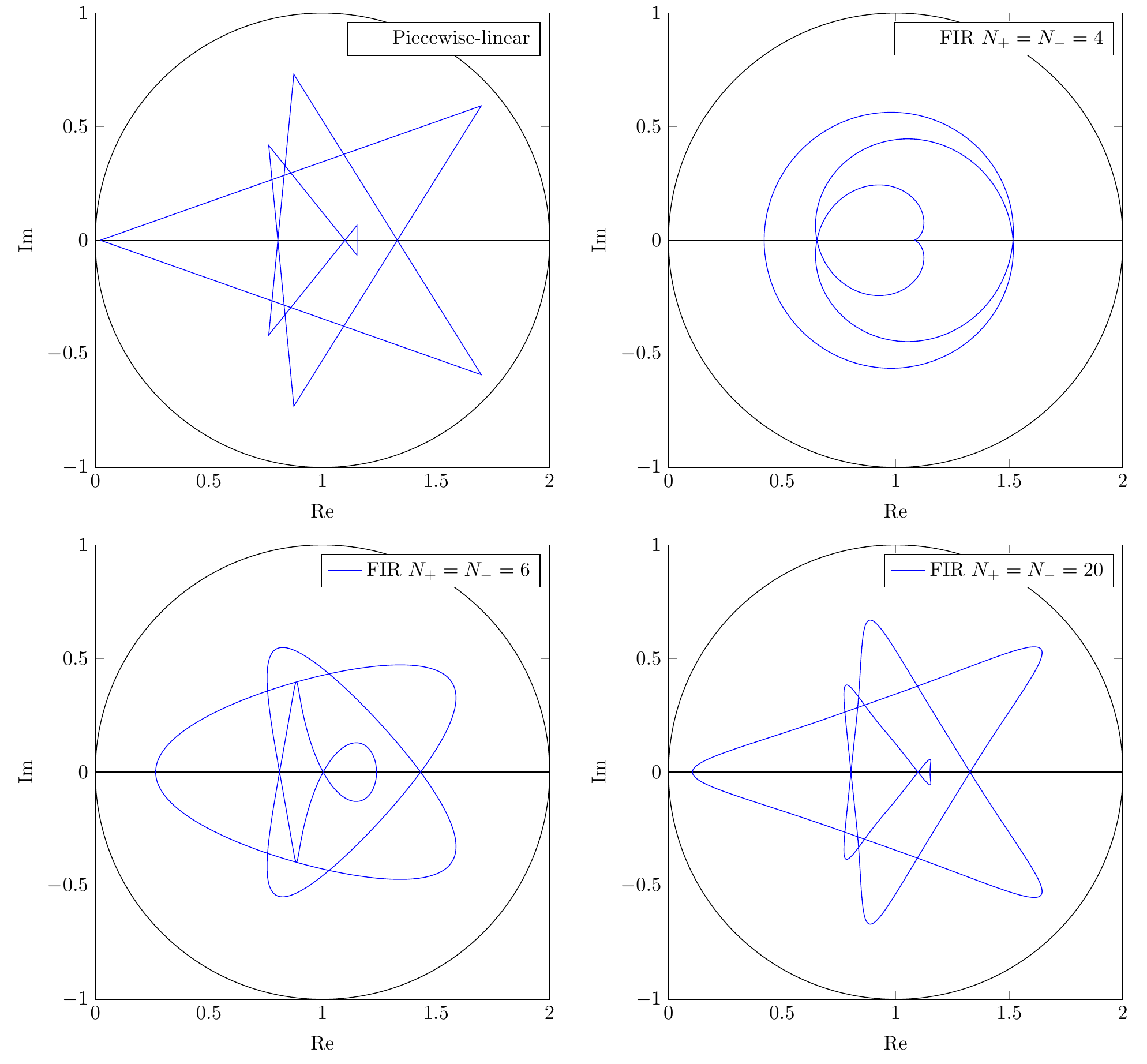}
		\caption{Nyquist plot of the piecewise-linear ZF-multiplier $H(z) = 1 - \sum_{k=-\infty}^\infty h_k z^k$ given by the p.d. sequence $c = \sqrt{N}V\alpha_*$ with optimal values $\alpha_*$ and $t_*$ in (\ref{eq:mainProgram}) for the plant $G_1(z) = \frac{-z}{z^2-1.8z+0.81}$, $\kappa = 0.3$ and $N = 9$ together with its FIR-approximations $1 - \sum_{k=-N_-}^{N_+} h_k z^k$.}
		\label{fig:pwlNyquist}
	\end{figure*}\\
	
	In order to further analyze the convex program (\ref{eq:pdProgram}) for a fixed $N \in \b{N}$,
	we parametrize the circulant matrix $(c_{j-k})_{j,k=0}^{N-1}$ by its eigenvalues \cite{gray2006toeplitz} and stack the coefficients $(c_j)_{j=0}^{N-1}$ as well as the eigenvalues $(\lambda_j)_{j=0}^{N-1}$ into column vectors $c$ and $\lambda$, respectively.
	This gives the relationship $c = V\lambda/\sqrt{N}$, where $V =(z_j^l)_{l,j=0}^{N-1}/\sqrt{N}$ is the DFT-matrix.
	Moreover, we define $Q = \operatorname{diag}(G(z_0),\ldots,G(z_{N-1}))$ and $e = (1,\ldots,1)^\top \in \b{R}^{N}$.
	
	With these ingredients, we consider the linear program
	\begin{align}
		\label{eq:mainProgram}
		\inf -t\text{\ \ s.t.\ }
		\Bigg\{\begin{split}
			t \in \b{R}\,, \alpha \in \b{R}^{N}\,, e^\top \alpha \leq 1\,,\alpha \geq 0\,, \\ \Re(Q(e-\sqrt{N}V\alpha)) \leq -t e
		\end{split}
	\end{align}
	{where we maximize the margin $t = \epsilon/2$ for the FDI \eqref{eq:positiveDefiniteFDI} evaluated at the $N$-th roots of unity.}
	Note that its optimal value $-t_*$ is nonpositive  since $\alpha = (1,0,\ldots,0)^\top$ and $t = 0$ are always feasible for \eqref{eq:mainProgram}.
	\begin{lemma}
		\label{thm:infeasibleImpliesOptimalValueBound}
		The optimal value $-t_*\leq 0$ of the program (\ref{eq:mainProgram}) satisfies $t_* \leq \delta$ iff (\ref{eq:pdProgram}) with strict inequalities is infeasible in the convex set of p.d. sequences $(c_j)_{j=0}^{N-1}$ with $c_0 \leq 1$.
	\end{lemma}
	
	If combining Lemma~\ref{thm:nonexistenceImpliesInfeasible} with
	Lemma~\ref{thm:infeasibleImpliesOptimalValueBound}, we have shown in this section that the nonexistence of a multiplier $M \in \c{H}_1$ satisfying \eqref{eq:negativeMultiplier} implies the existence of some $N \in \b{N}$ for which the optimal value $-t_*$ of \eqref{eq:mainProgram} satisfies $t_* \in [0,\eps]$.
	%for $\delta = \epsilon > \epsilon/2$.}\\

	%\magn{Honestly: Through the formulation of the intermediate steps by negation and minimization of $-t$ (instead of the maximization of $t$), this latter conclusion requires several re-readings by a careful reader. I expect complaints by reviewers at this point!
	%
	%Probably you do not agree.}
	
	\subsection{Construction of destabilizing nonlinearities}
	
	\label{sec:constructionUncertainty}
	
	%After having shown that the nonexistence of a suitable multiplier implies a bound on the optimal value $t_*$ of the linear program \eqref{eq:mainProgram},
	We now assume that (\ref{eq:mainProgram}) has an optimal value $-t_*$ with $t_* \in [0,\delta]$ and proceed to construct some nonlinearity which permits to prove
	Theorem~\ref{thm:mainNecessity}.
	
	For this purpose we consider the dual program of (\ref{eq:mainProgram}), which is given by
	\begin{align*}
		\sup \Re(\mu^\top Qe)-\eta \text{\ s.t.\ }
		\begin{cases}
			\mu \in \b{R}^N\,, \eta \in \b{R}\,,  \mu, \eta \geq 0 \\
			\sqrt{N} \Re(\mu^\top Q V) \leq \eta e^\top \\
			\mu^\top e = 1\,. \\
		\end{cases}
	\end{align*}
	By strong duality there exist $\mu \in \b{R}^N$ and $\eta \in \b{R}$ with
	\begin{gather}
		\label{eq:dualityConstraints}
		\begin{gathered}
			\eta -\Re(\mu^\top Qe) = t_* \leq \delta\,, \quad
			\mu^\top e = 1\,, \\
			-\sqrt{N}\Re(\mu^\top QV) + \eta e^\top \geq 0\,, \quad \mu \geq 0\,, \quad \eta \geq 0\,.
		\end{gathered}
	\end{gather}
	In order to extract the necessary information from these dual constraints, we define the vectors
	\begin{align}
		\label{eq:vectorsDual}
		y =  V\mu^{1/2} \in \b{R}^N \text{\ \ and\ \ } x = VQ\mu^{1/2} \in \b{R}^N
	\end{align}
	by the inverse DFT of $\mu^{1/2} = (\mu_{j}^{1/2})_{j=0}^{N-1}$ and $Q\mu^{1/2}$, respectively.
	Note that these are related by $x = Ty$, where $T = VQV^* \in \b{R}^{N \times N}$ is circulant.
	Now if we set
	\begin{align}
		\label{eq:matrixY}
		Y = \frac{1}{N}\sum_{k=0}^{N-1} S^k yy^\top S^{-k}\,,
	\end{align}
	then $Y$ is positive semi-definite and, hence, can be factored as $Y = U^\top U$ for some $U = \mat{\bar{y}_0 & \cdots & \bar{y}_{N-1}} \in \b{R}^{d \times N}$ and $d \leq N$.
	Denoting
	
	\begin{align}
		\label{eq:interpolationVectors}
		\begin{aligned}
			\bar{y} &= (\bar{y}_0^\top,\ldots,\bar{y}_{N-1}^\top)^\top \in (\b{R}^{d})^N\,, \\
			\bar{x} &= (\bar{x}_0^\top,\ldots,\bar{x}_{N-1}^\top)^\top := (T \otimes I_d)\bar{y} \in (\b{R}^{d})^N\,, \\
			&\bar{y}_N = \bar{x}_N := 0,
		\end{aligned}
	\end{align}
	we have the following result.
	
	\begin{lemma}
		\label{lem:interpolationMonotone}
		For the constructed $\bar{x}$, $\bar{y}$ (under the assumptions of this section),
		there exist some $\hat{x}_k, \hat{y}_k \in \b{R}^d$, $k=0,\ldots,N$, and $f \in \operatorname{slope}_d[0,\infty]$ such that $\operatorname{dom}(f) = \b{R}^d$, $0 \in f(0)$ and
		\begin{gather}
			\label{eq:distanceSignals}
			\begin{gathered}
				\hat{y}_k - \hat{y}_N \in f(\hat{x}_k - \hat{x}_N)\,, \\
				\norm{\bar{x}_k - \hat{x}_k}^2 \leq \rho\,,
				\norm{\bar{y}_k - \hat{y}_k}^2 \leq \rho\,,
			\end{gathered}
		\end{gather}
		for $k=0,\ldots,N$, where $\rho = t_*/N$.
		%In particular $0 \in f(0)$.
	\end{lemma}
	
	Next, since $\bar{x} = (T \otimes I_d)\bar{y}$ with
	\begin{align*}
		&(T \otimes I_d)= \\
		&\quad (V \otimes I_d)\operatorname{diag}(G(z_0) \otimes I_d,\ldots,G(z_{N-1}) \otimes I_d)(V \otimes I_d)^*\,,
	\end{align*} 
	
	it follows from Lemma \ref{lem:periodicLTI} that there is some initial state $\bar{\xi}_0$ of $G \otimes I_d$ such that $(\bar{x}_{k \mod N})_{k=0}^\infty$ is the output of $G \otimes I_d$ to the input $(\bar{y}_{k \mod N})_{k=0}^\infty$.
	Setting $\Delta = \Delta_f$,
	\begin{align*}
		\hat{w}_k = \hat{y}_k - \hat{y}_N\,, \quad
		\hat{z}_k = \hat{x}_k - \hat{x}_N\,, \\
		\tilde{x}_k = \hat{z}_k - \bar{x}_k\,, \quad
		\tilde{y}_k = \bar{y}_k - \hat{w}_k\,,
	\end{align*}
	for $k = 0,\ldots,N-1$ and by $N$-periodic continuation, we obtain signals that satisfy the loop equation corresponding to Fig.~\ref{fig:instableConfig}.
	\begin{figure}[h]%
		\centering
		\includegraphics[scale=1]{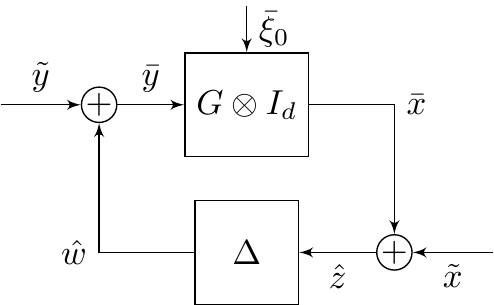}
		%	\subfloat[Modified feedback interconnection]{
		%		\includegraphics[scale=0.8]{mod_feedback.pdf} \label{fig:multFeedback}}
		\caption{Configuration for the constructed signals.}
		\label{fig:instableConfig}
	\end{figure}
	
	Moreover, from \eqref{eq:distanceSignals} and $\bar x_N=0$, $\bar y_N=0$ it follows that
	\begin{gather*}
		\quad\operatorname{pow}(\tilde{x})^2
		\leq \frac{1}{N}\sum_{k=0}^{N-1} (\norm{\bar{x}_k - \hat{x}_k} + \norm{\hat{x}_N})^2 \\
		\leq \frac{2}{N}\sum_{k=0}^{N-1} (\norm{\bar{x}_k - \hat{x}_k}^2 + \norm{\bar{x}_N-\hat{x}_N}^2)
		\leq 4 \rho
		= 4 \frac{t_*}{N}
	\end{gather*}
	and analogously $\operatorname{pow}(\tilde{y})^2 \leq 4t_*/N$,
	while $N\operatorname{pow}(\bar{y})^2 =\tr(U^\top U) = \tr(Y) = \tr(yy^\top) = \<y,y\> = \mu^\top e = 1$.
	
	\subsection{Proof of Theorem \ref{thm:mainNecessity}}
	\label{sec:proofMain}
	Now we can present the proof of our main result.\\
	\begin{proof}
		Assume that there is no $M \in \c{H}_1$ such that \eqref{eq:negativeMultiplier} is satisfied.
		%Hence by Theorem \ref{thm:nonexistenceImpliesInfeasible} we see that there is some $N \in \b{N}$ such that (\ref{eq:pdProgram}) with strict inequalities is infeasible for $\delta = \epsilon > \epsilon/2$.
		%Then, by Theorem \ref{thm:infeasibleImpliesOptimalValueBound}, the optimal value $t_*$ of the program (\ref{eq:mainProgram}) satisfies $t_* \in [0,\delta] = [0,\epsilon]$.
		From Section \ref{sec:linearProgram} we infer $t_* \in [0,\epsilon]$, where $-t_*$ is the optimal value of \eqref{eq:mainProgram}.
		Following Section \ref{sec:constructionUncertainty} we obtain some $\Delta \in \bs{\Delta}_d(0,\infty)$, periodic signals $\tilde{x}, \tilde{y}$ as well as some inital state $\bar{\xi}_0$ of $G \otimes I_d$ such that the loop equations corresponding to Fig. \ref{fig:instableConfig} are satisfied.
		The latter is equivalent to the same configuration where $\tilde{x}$ is replaced by $\tilde{x} + \tilde{t}$, with $\tilde{t}_k = (CA^k \otimes I_d)\bar{\xi}_0$, and $\bar{\xi}_0$ is replaced by $0$.
		Since $G$ is stable, $A$ is Schur and therefore $\operatorname{pow}(\tilde{x}+\tilde{t}) = \operatorname{pow}(\tilde{x}) \leq 4t_*/N$.
		Thus, by Lemma \ref{lem:powerGainRelation}, the gain of the interconnection \eqref{eq:iqcRelationInterconnection} of $G \otimes I_d$ and $\Delta$ is at least
		\begin{align*}
			\frac{\operatorname{pow}(\bar{y})}{\sqrt{\operatorname{pow}(\tilde{x}+\tilde{t})^2 + \operatorname{pow}(\tilde{y})^2}}
			\geq \frac{1/\sqrt{N}}{\sqrt{8t_*/N}}
			= \frac{1}{\sqrt{8t_*}}
			\geq \frac{1}{\sqrt{8\epsilon}}\,.
		\end{align*}
		The result for $\kappa\in(0,\infty)$
		then follows from a standard loop transformation and can be found in \cite{kharitenko2022some}.
	\end{proof}
	
	\begin{remark}
		\label{rem:infiniteGain}
		If $t_* = 0$ holds for \eqref{eq:mainProgram}, the above construction shows that we can find nonzero internal $N$-periodic oscillations and the gain of the interconnection is actually infinite.
	\end{remark}
	
	\begin{remark}
		\label{rem:hilbertSpace}
		If we factor $Y$ defined in \eqref{eq:matrixY} as $Y = U^\top U$ for some operator $U\in \c{L}(\b{R}^N,\ell^2)$ (i.e. $d = \infty$) and proceed with our construction, we obtain the following result: If there is no $\epsilon > 0$ and $M \in \c{H}_1$ with \eqref{eq:negativeMultiplier}, then for each $K > 0$ there is some $\Delta \in \bs{\Delta}_{\infty}(0,\kappa)$ (where $\bs{\Delta}_{\infty}(0,\kappa)$ denote the slope-restricted nonlinearities on the Hilbert space $\ell^2$ \cite{bauschke2011convex}) such that the gain of the interconnection between $G \otimes I_{\ell^2}$ and $\Delta$ is at least $K$, i.e.
		the criterion in Theorem \ref{thm:stabilityIQC} is exact for systems of the form $G \otimes I_{\ell^2}$.
	\end{remark}
	
	\section{Numerical Example}
	
	Consider the plant
	\begin{align*}
		G_d(z) = \frac{1.1z + 0.6}{z^2+1.8z+0.9} \otimes I_d
	\end{align*}
	and $\Delta \in \bs{\Delta}_d(0,\kappa)$.
	The Nyquist value is given by $\kappa = 2.17$.
	
	Moreoever, there exists some multiplier proving that the interconnection \eqref{eq:iqcRelationInterconnection} is stable for all $d$ if $\kappa \leq 1.86$.
	However if $\kappa > 1.86$, then \eqref{eq:mainProgram} has the optimal value $t_* = 0$ for $N = 5$ and hence there exists no corresponding multiplier satisfying \eqref{eq:negativeMultiplier} for any $\epsilon > 0$.
	In particular, by Remark \ref{rem:infiniteGain}, we can construct some $\Delta \in \bs{\Delta}_5(0,\kappa)$ such that the interconnection of $G_5$ and $\Delta$ has infinite gain.
	Note that for $\kappa > 2.15$ the original interconnection ($d = 1$) is unstable and a corresponding nonlinearity $\Delta \in \bs{\Delta}_1(0,\kappa)$ can be constructed by the method in \cite{seiler2020construction}.

	\section{Discussion}
	
	\label{sec:discussion}
	
	\subsection{SISO exactness and results for SSV}
	
	The question about whether the multipliers are exact for the original SISO interconnection remains unanswered by Theorem \ref{thm:mainNecessity}, since, in general, $d$ will be larger than $1$.
	It is easy to see that if the interconnection between $G \otimes I_d$ and $\bs{\Delta}_d(\lambda,\kappa)$ is robustly stable for some $d$, then so is the interconnection between $G$ and $\bs{\Delta}_1(\lambda,\kappa)$.
	The converse may or may not be true.
	
	In \cite{seiler2020construction} it is shown that if for some coprime $\alpha, \beta \in \b{N}_0$ such that $\alpha < \beta$ the main plant $G$ satisfies the phase constraint
	\begin{align}
		\label{eq:phaseConstraint}
		-\frac{\pi}{N} \leq \arg(G(\exp(i\pi \tfrac{\alpha}{\beta})) \leq \frac{\pi}{N}\,,
	\end{align}
	where $N = 2\beta$ if $\alpha$ is odd and $N = \beta$ if $\alpha$ is even, a nonlinearity $\Delta \in \bs{\Delta}_1(0,\kappa)$ can be constructed such that the interconnection \eqref{eq:iqcRelationInterconnection} has infinite gain.
	One can show that (\ref{eq:phaseConstraint}) implies $t_* = 0$ in \eqref{eq:mainProgram}.
	This reveals that our ``instability criterion'' in Remark \ref{rem:infiniteGain} is not less conservative than the one given in \cite{seiler2020construction}.
	But it is the construction of a SISO nonlinearity ($d = 1$)
	in \cite{seiler2020construction} that makes this approach different from ours.
	
	The procedure to lift the interconnection from $G$ to $G \otimes I_d$ is reminiscent of a result shown in \cite{bercovici1996structured, ball2016bounded} for the structured singular value (SSV).
	To state it, we define for $M \in \b{C}^{n \times n}$ and the complex block structure $\bs{\Delta} = \{\operatorname{diag}(\delta_1 I_{n_1},\ldots,\delta_{r} I_{n_r},\Delta_{r+1},\ldots,\Delta_{r+c}) \mid \delta_j \in \b{C}\,, \Delta_{j} \in \b{C}^{n_j \times n_j}\} \subseteq \b{C}^{n \times n}$ the SSV and its convex upper bound:
	\begin{align*}
		\mu_{\bs{\Delta}}(M) &= \inf\{\norm{\Delta} \mid \Delta \in \bs{\Delta}\,, -1 \in \sigma(M\Delta)\}\,, \\
		\hat{\mu}_{\bs{\Delta}}(M) &= \inf\{\norm{XMX^{-1}} \mid X \in \bs{\Delta}'\,, X \text{\ invertible}\}\,.
	\end{align*}
	Here, $\bs{\Delta}' = \{X \in \b{C}^{n \times n} \mid \forall \Delta \in \bs{\Delta}:\;X\Delta = \Delta X\}$.
	\begin{theorem}[Theorem 4.1, \cite{bercovici1996structured}]
		For $d = \sum_{j=1}^r n_j^2 + \sum_{j=r+1}^{r+c} n_j$ and $\bs{\Delta}_{d} = \{(\Delta_{ij})_{i,j=1}^d \mid \Delta_{ij} \in \bs{\Delta}\}$,
		\begin{align*}
			\hat{\mu}_{\bs{\Delta}}(M) = \mu_{\bs{\Delta}_{d}}(I_d \otimes M)\,.
		\end{align*}
	\end{theorem}
	Since it is well-known that, in general, $\mu_{\bs{\Delta}}(M) \neq \hat{\mu}_{\bs{\Delta}}(M)$, we conjecture that the criterion in Theorem \ref{thm:stabilityIQC} is not necessary for robust stability of \eqref{eq:iqcRelationInterconnection} with $d = 1$.

	\subsection{Linear Program \eqref{eq:mainProgram}}
	
	The linear program \eqref{eq:mainProgram} appeared in other publications.
	The following result is given in \cite{zhang2021duality}.
	
	\begin{theorem}[Proposition 1, \cite{zhang2021duality}]
		Let $G$ be a SISO, bounded, LTI system.
		If there exists some $N \in \b{N}$ and $\mu \in \b{R}^N$ with $\mu\neq 0$,
		$\mu \geq 0$ such that
		\begin{align}
			\label{eq:nonexistenceConditionZhang}
			\sum_{j=0}^{N-1} \mu_{j} \Re(G(z_j)(1-z_j^{-k})) \geq 0
		\end{align}
		holds for all $k=0,\ldots,N-1$, then there exists no OZF-multiplier $M$ with $\Re(G(z)M(z)) < 0$ for all $z \in \b{T}$.
	\end{theorem}
	
	If the program (\ref{eq:mainProgram}) has the optimal value $t_* = 0$, then
	(\ref{eq:nonexistenceConditionZhang}) is indeed satisfied, as extracted from
	(\ref{eq:shiftConditionsFrequency}).
	Conversely, suppose that (\ref{eq:nonexistenceConditionZhang}) is satisfied for some $\mu \geq 0$, $\mu \neq 0$. If the additional unbiasedness condition $\sum_{j=0}^{N-1} \mu_j \Re(G(z_j)) \geq 0$ holds, we obtain (\ref{eq:shiftConditionsFrequency}) with $t_* = 0$ (after modifying $\mu$
	to satisfy $e^\top \mu = 1$ and $\mu_{N-j} = \mu_j$ for all $j=1,\ldots,N-1$ without loss of generality.)
	%then we can assume without loss of generality that $e^\top \mu = 1$ as well as $\mu_{N-j} = \mu_j$ for all $j=1,\ldots,N-1$ to obtain (\ref{eq:shiftConditionsFrequency}) with $t_* = 0$
	Then we can proceed as in Section \ref{sec:constructionUncertainty} to obtain some $d \leq N$ and construct a destabilizing nonlinearity in $\bs{\Delta}_d(0,\infty)$, i.e., for which the interconnection with $G \otimes I_d$ has infinite gain.
	
	As another interesting piece of work, the linear program \eqref{eq:mainProgram} also appeared in \cite{claasen1975frequency}.
	
	\begin{theorem}[Theorem 2,\cite{claasen1975frequency}]
		\label{thm:auto-oscillations}
		Let $G$ be a SISO, bounded, LTI system and suppose that
		(\ref{eq:mainProgram}) has the optimal value $t_* = 0$.
		Then the interconnection \eqref{eq:iqcRelationInterconnection}
		of $G$ and any $\Delta \in \Delta_1(0,\infty)$ has no nonzero $N$-periodic internal oscillations, i.e., there do not exist $N$-periodic $e_1,e_2\neq 0$ and an initial state $\xi_0$ of $G$ such that \eqref{eq:iqcRelationInterconnection} is satisfied with $u_1,u_2 = 0$.
	\end{theorem}
	
	In view of Lemma \ref{thm:infeasibleImpliesOptimalValueBound}, we can interpret the existence of some OZF-multiplier $M$ with $\Re G(z_j)M(z_j) < 0$  for $j=0,\ldots,N-1$ as a criterion for the nonexistence of $N$-periodic internal oscillations.
	If $\Re G(z)M(z) < 0$ holds for all $z \in \b{T}$, then there are no internal oscillations of any period, which can be (non-rigorously) thought of as implying stability of the interconnection \eqref{eq:iqcRelationInterconnection}.
	
	\section{Conclusion}
	
	In this note it was shown that the robust stability test induced by O'Shea-Zames-Falb multipliers is exact if the interconnection structure is extended by lifting.
	An explicit method for the construction of the destabilizing nonlinearities as well as connections to duality bounds and criteria for absence of internal periodic oscillations were presented.
	Exactness of the test for the original interconnection remains open and we conjecture it to be false, based on analogous results for the structured singular value.
	
	\setcounter{section}{0}
	\renewcommand{\thesection}{\Alph{section}}
	
	\section{Appendix}
	\subsection{Technical Proofs}
	
	%\red{In the following we present the proofs of Lemma \ref{thm:nonexistenceImpliesInfeasible}, Lemma \ref{thm:infeasibleImpliesOptimalValueBound} and Lemma \ref{lem:interpolationMonotone}.
	%The proof of the latter contains an explicit construction of the destabilizing nonlinearity.}
	
	\label{sec:proofs}
	\subsubsection{Proof of Lemma \ref{thm:nonexistenceImpliesInfeasible}}
	
	\begin{proof}
		Assume that $\delta > \epsilon/2$ and that (\ref{eq:pdProgram}) is feasible for each $N \in \b{N}$, i.e.,
		\begin{align*}
			\Re(G(z_j^{N})(1-c_j^{N})) \leq -\delta \text{\ for all\ } j=0,\ldots,N-1
		\end{align*}
		with some p.d. $(c_j^N)_{j=0}^{N-1}$ and $c_0^N \leq 1$; recall that $z_0^{N},\ldots,z_{N-1}^N$,
		$z_N^N = z_0^N$ are the $N$-th roots of unity and the superscript $N$ should not be confused with exponentiation.
		
		For each $N \in \b{N}$ let $f_N:\b{T} \to \b{C}$ be the piecewise-linear interpolant of $(c_j^N)_{j=0}^{N-1}$ at the nodes $\{z_0^{N},\ldots,z_{N-1}^N\}$.
		Then $f_N$ is continuous with $f_N(1) \leq 1$ and p.d. by Theorem \ref{thm:belov1}.
		
		Let $\eta = \delta - \epsilon/2 > 0$.
		Since $G$ has a finite-dimensional state-space representation, $G(\cdot)$ is (uniformly) continuous on $\b{T}$.
		Therefore there exists some $\lambda > 0$ such that if $z,w \in \b{T}$ are taken with $\abs{z-w} < \lambda$, then $\abs{G(z) - G(w)} < \eta/4$.
		
		Now fix a large $N \in \b{N}$ with $2\pi/N < \lambda$. We then clearly have $\abs{z_j^N - z_{j+1}^N} < \lambda$ for each $j=0,\ldots,N-1$.
		Defining the arc-function between $z_j^N$ and $z_{j+1}^N$ for $t \in [0,1]$ by
		\begin{align*}
			z_j^N(t) := \exp\left(i\left(\frac{2\pi j}{N}(1-t) + \frac{2\pi (j+1)}{N}t\right)\right)\,,
		\end{align*}
		we get $\abs{z_j^N(t) - z_{j}^N} < \lambda$ and $\abs{z_j^N(t) - z_{j+1}^N} < \lambda$ for any $t \in [0,1]$ and $j=0,\dots,N-1$.
		Also note that $f_N(z_j^N(t)) = (1-t)c_j^N + t c_{j+1}^N$ by the mere definition of $f_N$.
		Thus for $t \in [0,1]$ and $j=0,\dots,N-1$ we infer
		\begin{align*}
			&\Re\left[G(z_j^N(t))(1-f_N(z_j^N(t))\right] \\
			&= (1-t)\Re\left[G(z_j^N(t))(1-c_j^N)\right] \\
			&\qquad + t \Re\left[G(z_j^N(t))(1-c_{j+1}^N)\right] \\
			&= (1-t)\Re\left[G(z_j^N)(1-c_j^N)\right] \\
			&\qquad-(1-t)\Re\left[(G(z_j^N(t)) - G(z_j^N))(1-c_{j}^N)\right] \\
			&\qquad + t \Re\left[G(z_{j+1}^N)(1-c_{j+1}^N)\right] \\
			&\qquad- t\Re\left[(G(z_j^N(t)) - G(z_{j+1}^N))(1-c_{j+1}^N)\right] \\
			&\leq -(1-t)\delta + (1-t)\abs{\Re\left[(G(z_j^N(t)) - G(z_j^N))(1-c_{j}^N)\right]} \\
			&\quad - t\delta + t\abs{\Re\left[(G(z_j^N(t)) - G(z_{j+1}^N))(1-c_{j+1}^N)\right]} \\
			&\leq -\delta + (1-t) \abs{G(z_j^N(t)) - G(z_j^N)} \abs{1-c_{j}^N} \\
			&\qquad +t  \abs{G(z_j^N(t)) - G(z_{j+1}^N)} \abs{1-c_{j+1}^N}\\
			&\leq -\delta + 4\frac{\eta}{4}
			= -\delta + \eta
			= -\frac{\epsilon}{2}\,;
		\end{align*}
		we have used the fact that $\Re\left[G(z_{j}^N)(1-c_{j}^N)\right] \leq -\delta$ and $\Re\left[G(z_{j+1}^N)(1-c_{j+1}^N)\right] \leq -\delta$
		by the definition of $(c_j^N)_{j=0}^{N-1}$ as well as that $\abs{G(z_j^N(t)) - G(z_j^N)} < \eta/4$ and $\abs{G(z_j^N(t)) - G(z_{j+1}^N)} < \eta/4$ and also that $\abs{1-c_j^N} \leq 2$ and $\abs{1-c_{j+1}^N} \leq 2$ since $c_0^N \leq 1$ and $(c_j^N)_{j=0}^{N-1}$ is p.d..
		
		Since all the considered arcs cover the unit circle,
		we have shown that $f = f_N$ actually satisfies (\ref{eq:positiveDefiniteFDI}),
		which implies (see Section~\ref{sec:linearProgram}) that there exists some $M \in \c{H}_1$ satisfying \eqref{eq:negativeMultiplier}, a contradiction.
	\end{proof}
	
	\subsubsection{Proof of Lemma \ref{thm:infeasibleImpliesOptimalValueBound}}
	
	\begin{proof}
		If we identify $\lambda/N = \alpha$, then $c$ equals $\sqrt{N} V \alpha$ and the condition $e^\top \alpha \leq 1$ translates to $c_0 \leq 1$, the condition $\Re(Q(e-\sqrt{N}V\alpha)) \leq -te$ to $\Re((1-c_j)G(z_j)) \leq -t$ for $j=0,\ldots,N-1$, and $\alpha \geq 0$ to $C = (c_{j-k})_{j,k=0}^{N-1} \geq 0$, i.e., to the fact that $(c_j)_{j=0}^{N-1}$ is p.d.
		
		If \eqref{eq:pdProgram} with strict inequalities is feasible for some p.d. sequence $(c_j)_{j=0}^{N-1}$, then we infer that $\alpha = \lambda/N \geq 0$ and $-t = \max_{j} \Re((1-c_j)G(z_j)) < \delta$ are feasible for \eqref{eq:mainProgram} so that $-t_* \leq -t < -\delta$, i.e, $t_*>\delta$.
		
		Conversely, if $t_* >\delta$ and one picks any $t \in (\delta,t_*)$, there is some $\alpha^* \geq 0$ with $\Re(Q(e-\sqrt{N}V\alpha^*)) \leq -t e<-\delta e$ and $e^\top \alpha^* \leq 1$.
		Thus $c = \sqrt{N}V\alpha^*$ is a feasible solution of \eqref{eq:pdProgram} with strict inequalities.
	\end{proof}
	
	\subsubsection{Proof of Lemma \ref{lem:interpolationMonotone}}
	
	This proof requires some preparations as follows.
	\begin{defn}
		A relation $f \subseteq \b{R}^d \times \b{R}^d$ is $\epsilon$-approximately cyclically monotone if for each $n \in \b{N}_0$ and pairs $(x_1,y_1),\ldots,(x_n,y_n) \in f$ we have
		\begin{align}
			\label{eq:approxCyclicallyMonotoneDef}
			\sum_{k=1}^n \<y_{k},x_{k+1} - x_k\> \leq n \epsilon
		\end{align}
		where $x_{n+1}:=x_1$.
	\end{defn}
	This is a generalization of cyclical monotonicity and the following theorem generalizes \cite[Theorem 12.25]{rockafellar2009variational}.
	
	\begin{theorem}
		\label{thm:cyclicallyMonotoneSubgradient}
		A relation $f \subseteq \b{R}^d \times \b{R}^d$ is $\epsilon$-approximately cyclically monotone iff $f \subseteq \partial_\epsilon F$ holds for some convex, lsc and proper function
		$F:\b{R}^d \to \b{R} \cup \{\infty\}$.
	\end{theorem}
	
	In here, $\partial_\epsilon F(x) = \{(x,y) \in \b{R}^d\times \b{R}^d\mid \forall u\in \b{R}^d:\;F(u) - F(x) \geq \<y,u-x\> - \epsilon\}$ is the $\epsilon$-subdifferential of $F$.
	
	\begin{proof}
		To prove the nontrivial direction, we fix some $(x_0,y_0) \in f$ and set
		\begin{align*}
			F(x) = \sup \left\{\<y_n,x-x_n\> + \sum_{k=0}^{n-1} \<y_k,x_{k+1}-x_k\> - n \epsilon\right\}
		\end{align*}
		with the supremum taken over all $(x_1,y_1),\ldots,(x_n,y_n) \in f$, $n \in \b{N}_0$.
		Then $f \subseteq \partial_\epsilon F$ can be shown (see \cite{kharitenko2022some}) by following the original proof in \cite[Theorem 12.25]{rockafellar2009variational}.
	\end{proof}
	
	Finally we recall \cite[Theorem 3.1.2]{zalinescu2002convex} which relates the $\epsilon$-subgradient to the subgradient of a convex function:
	
	\begin{theorem}[Bronsted-Rockafellar]
		\label{thm:bronstedRockafellar}
		Let $F:\b{R}^d \to \b{R}$ be convex, %lsc and proper,
		$\epsilon \geq 0$ and $y \in \partial_\epsilon F(x)$.
		Then there exist $\bar{x},\bar{y} \in \b{R}^d$ such that $\bar{y} \in \partial F(\bar{x})$ and $\norm{x - \bar{x}}^2 \leq \epsilon$ and $\norm{y - \bar{y}}^2 \leq \epsilon$.
	\end{theorem}
	
	Now we can proceed to the proof of Lemma \ref{lem:interpolationMonotone}.
	
	\begin{proof}
		The proof is divided into four parts.\\
		\emph{Part 1.} The vectors $y$ and $x$ defined in \eqref{eq:vectorsDual} satisfy the cross-correlation inequalities
		\begin{align}
			\label{eq:crossCorrelation}
			\begin{aligned}
				\<x,y\> \geq -t_* \text{\ and\ }
				\<S^k x,y\> \leq \<x,y\> + t_*
			\end{aligned}
		\end{align}
		for $k=0,\ldots,N-1$, where $S = (\delta_{i-j+1 \mod N})_{i,j=0}^{N-1}$ is the circulant shift-matrix.
		Indeed, the first entry of the inequality constraint in \eqref{eq:dualityConstraints} shows $t_* = \eta - \Re(\mu^\top Qe) \geq 0$.
		If $\mu = (\mu_0,\mu_1,\ldots,\mu_{N-1})^\top$ satisfies \eqref{eq:dualityConstraints},
		note that so does $\tilde{\mu} = (\mu_0,\mu_{N-1},\ldots,\mu_1)^\top$ since $G(z_j) = G(z_{N-j})^*$.
		By averaging $\mu$ and $\tilde{\mu}$ we can hence assume that $\mu$ has the property $\mu_j = \mu_{N-j}$ for $j=1,\ldots,N-1$.
		Thus we can omit the real parts. By elimination of $\eta$ and $V$ in \eqref{eq:dualityConstraints} we get
		\begin{align}
			\label{eq:shiftConditionsFrequency}
			\begin{aligned}
				-t_* &\leq \sum_{j=0}^{N-1} \mu_j G(z_j) \,,\\
				\sum_{j=0}^{N-1} \mu_j G(z_j)z_{j}^k &\leq \sum_{j=0}^{N-1} \mu_j G(z_j) + t_*
			\end{aligned}
		\end{align}
		for $k=0,\ldots,N-1$. Then, by $\mu\geq 0$ and the fact that $z_k^j = z_j^k$, the conditions (\ref{eq:shiftConditionsFrequency}) are precisely
		\begin{align*}
			\begin{aligned}
				-t_* &\leq \<Q\mu^{1/2},\mu^{1/2}\>,\\
				\<Z^k Q\mu^{1/2},\mu^{1/2}\> &\leq \<Q\mu^{1/2},\mu^{1/2}\> + t_*
			\end{aligned}
		\end{align*}
		for $k=0,\ldots,N-1$,
		where $Z = \operatorname{diag}(z_0,\ldots,z_{N-1})$.
		Applying the DFT yields \eqref{eq:crossCorrelation}, since $S = VZV^*$.\\
		
		\emph{Part 2.} The matrix $Y$, as defined in \eqref{eq:matrixY}, satisfies $\tr\left(MTY\right) \geq -\tr(M)t_*/N$ for all $M \in \b{R}^{N \times N}$ that are doubly hyperdominant.
		
		To see this, observe that \eqref{eq:crossCorrelation} imply
		\begin{align*}
			\tr(Tyy^\top) = \<Ty,y\> &\geq -t_*\,, \\
			\tr((I-S^k)Tyy^\top) = \<(I-S^k) Ty,y\> &\geq -t_*
		\end{align*}
		for $k=0,\ldots,N-1$.
		This can be equivalently expressed as
		\begin{align*}
			\tr\left(\left(\beta_0 I + \sum_{k=1}^{N-1} \beta_k(I-S^k)\right)Tyy^\top\right) \geq -\left(\sum_{k=0}^{N-1} \beta_k\right)t_*
		\end{align*}
		for all $\beta_0,\ldots,\beta_{N-1} \geq 0$.
		It is easy to see that $M = \beta_0 I + \sum_{k=1}^{N-1} \beta_k(I-S^k) = (\sum_{k=0}^{N-1} \beta_k) I - \sum_{k=1}^{N-1} \beta_k S^k$ is circulant and doubly hyperdominant.
		Conversely, every circulant and doubly hyperdominant matrix can be represented in this way.
		Hence we have shown that
		\begin{align*}
			\tr\left(MTyy^\top\right) \geq -\tr(M)t_*/N
		\end{align*}
		holds for all $M \in \b{R}^{N \times N}$ that are \emph{circulant and doubly hyperdominant}.
		
		Obviosuly the above inequality remains true if $yy^\top$ is replaced with $Y$.
		Now if $M$ is doubly hyperdominant (but not necessarily circulant), then, as $T$, $Y$ and $\bar{M} = \frac{1}{N}\sum_{k=0}^{N-1}S^kMS^{-k}$ are all circulant and $\bar{M}$ is again doubly hyperdominant, we have
		\begin{gather*}
			\tr\left(MTY\right)
			= \frac{1}{N}\sum_{k=0}^{N-1} \tr\left(S^kMTYS^{-k}\right)\\
			= \frac{1}{N}\sum_{k=0}^{N-1} \tr\left(S^kMS^{-k}TY\right)
			= \tr(\bar{M}TY) \\
			\geq -\tr(\bar{M})t_*/N
			= -\tr(M)t_*/N\,.
		\end{gather*}
		
		\emph{Part 3.}
		For the vectors $\bar{x}$, $\bar{y}$ defined in \eqref{eq:interpolationVectors},  the set
		\begin{align}
			\label{eq:approximateCyclicallyMonotoneSet}
			W = \{(\bar{x}_k, \bar{y}_k) \mid k=0,\ldots,N\}
		\end{align}
		of vector pairs is $t_*/N$-approximately cyclically monotone.
		Indeed, due to Part 2.,
		\begin{align}
			\label{eq:permIdentityTrace}
			-\tr\left((I-\Pi)TY\right) \leq \frac{1}{N}\tr(I-\Pi)t_*
		\end{align}
		for any subpermutation matrix $\Pi \in \b{R}^{N \times N}$.
		Obviously
		\begin{align}
			\label{eq:permIdentity}
			\tr(\Pi TY)
			= \<\bar{y},(\Pi \otimes I_d)(T\otimes I_d)\bar{y}\>.
		\end{align}
		%\magc{where $\<\bar{w},\bar{z}\> = \sum_{k=1}^N \<\bar{w}_k,\bar{z}_k\>$ for $\bar{w},\bar{z} \in (\b{R}^d)^{N}$.}
		If $i_1,\ldots,i_k \in \{0,\ldots,N\}$ are pairwise distinct, then we have two cases:
		\begin{enumerate}
			\item $i_j \neq N$ for all $j=1,\ldots,k$.
			Then let $\Pi$ be the matrix of the permutation $\pi = (i_1\ldots i_k)$ (i.e. $\pi(i_j) = i_{j+1}$ with $i_{k+1} = i_1$ and $\pi(l) = l$ for $l \notin \{i_1,\ldots,i_k\}$).
			By (\ref{eq:permIdentity}), (\ref{eq:permIdentityTrace}), and the fact that $\tr(I-\Pi) = k$ we infer
			\begin{gather*}
				\sum_{j=1}^{k} \<\bar{y}_{i_j},\bar{x}_{i_{j+1}} - \bar{x}_{i_j}\>
				= \sum_{j=1}^{k} \<\bar{y}_{i_j},\bar{x}_{\pi(i_j)} - \bar{x}_{i_{j}}\>  \\
				=  \sum_{l=0}^{N-1} \<\bar{y}_{l},\bar{x}_{\pi(l)} - \bar{x}_{l}\>
				= \<\bar{y},(\Pi \otimes I_d)\bar{x} - \bar{x}\> \\
				= \<\bar{y},((\Pi-I) \otimes I_d)(T \otimes I_d) \bar{y}\> \\
				= -\tr((I-\Pi)TY)
				\leq \tr(I-\Pi) t_*/N
				= k t_*/N\,.
			\end{gather*}
			
			\item $i_m = N$ for some $m \in \{1,\ldots,k\}$.
			Define the subpermutation matrix by $\Pi_{i_j i_{j+1}} = 1$ for $j=1,\ldots,k$ with $j \notin \{m-1,m\}$, $\Pi_{ll} = 1$ for $l \notin \{i_1,\ldots,i_k\}$ and $\Pi_{ij} = 0$ otherwise.
			By a similar argument
			\begin{gather*}
				\sum_{j=1}^{k} \<\bar{y}_{i_j},\bar{x}_{i_{j+1}} - \bar{x}_{i_{j}}\>
				= \sum_{l=0}^{N-1} \<\bar{y}_{l},((\Pi \otimes I_d)\bar{x})_{l} - \bar{x}_{l}\> \\
				= -\tr((I-\Pi)TY)
				\leq \tr(I-\Pi)t_*/N
				= k t_*/N\,,
			\end{gather*}
			since again $\tr(I-\Pi) = k$.
		\end{enumerate}
		\emph{Part 4.} Now we are in the position to construct the nonlinearity.
		Applying Theorem \ref{thm:cyclicallyMonotoneSubgradient} to the finite set $W$ yields a convex
		%, lsc and proper
		function $\bar{F}:\b{R}^d \to \b{R}$ with %$\operatorname{dom}(\bar{F}) = \b{R}^d$
		$\bar{y}_k \in \partial_{\rho} \bar{F} (\bar{x}_{k})$ for all $k=0,\ldots,N$,
		where $\rho = t_*/N$.
		
		By Theorem \ref{thm:bronstedRockafellar}, for each $k=0,\ldots,N$ there is a pair $(\hat{x}_k,\hat{y}_k) \in \b{R}^d \times \b{R}^d$ such that
		\begin{align*}
			\norm{\bar{x}_k - \hat{x}_k}^2 \leq \rho\,,\quad
			\norm{\bar{y}_k - \hat{y}_k}^2 \leq \rho\,,\quad
			\hat{y}_k \in \partial \bar{F}(\hat{x}_k)\,.
		\end{align*}
		Finally let $F(x) = \bar{F}(x+\hat{x}_N) - \<\hat{y}_N,x\>$ for $x\in\b{R}^d$.
		Then $F$ is convex and satisfies %, lsc, and proper with $\operatorname{dom}(F) = \b{R}^d$.
		Also $0 \in \partial F(0)$.
		Setting $f = \partial F$ yields the desired multifunction.
	\end{proof}

	\subsection{Auxiliary results}
	In this section we collect two well-known auxiliary results that we need in this note.
	\label{sec:auxiliary}
	\begin{lemma}
		\label{lem:powerGainRelation}
		If $R \subseteq \ell_{m}^{2e} \times \ell_{l}^{2e}$ and $(x,y) \in R$ is such that  $\operatorname{pow}(y) > 0$, then $\norm{R} \geq \frac{\operatorname{pow}(y)}{\operatorname{pow}(x)}$, where $1/0 := \infty$.
	\end{lemma}
	
	\begin{proof}
		If $\gamma > \norm{R}$, then there is some $\beta \in \b{R}$ with $\norm{P_N y} \leq \gamma \norm{P_N x} + \beta$ for all $T \in \b{N}_0$.
		Dividing by $\sqrt{N}$ and taking $\limsup_{N \to \infty}$ on both sides yields $\operatorname{pow}(y) \leq \gamma \operatorname{pow}(x)$.
		For $\operatorname{pow}(x) = 0$, this is a contradiction, which shows $\norm{R} = \infty$.
		For $\operatorname{pow}(x) > 0$ this implies $\gamma \geq \frac{\operatorname{pow}(y)}{\operatorname{pow}(x)}$.
	\end{proof}
	
	\begin{lemma}
		\label{lem:periodicLTI}
		Let $G:\ell_m^{2e} \to \ell_l^{2e}$ be LTI with a finite-dimensional state space representation $(A,B,C,D)$ such that $A \in \b{R}^{n \times n}$ and $\sigma(A) \cap \b{T} = \emptyset$.
		If $T = (V \otimes I_l)\operatorname{diag}(G(z_0),\ldots,G(z_{N-1}))(V \otimes I_m)^*$ and $x = Ty$ for some $y \in (\b{R}^{m})^N$, $x \in (\b{R}^{l})^N$, then
		there exists some initial state $\xi_0 \in \b{R}^n$ such that the periodic signal $\tilde{x} = (x_{k \mod N})_{k=0}^\infty$ is the output of $G$ to the periodic input $\tilde{y} = (y_{k \mod N})_{k=0}^\infty$.
	\end{lemma}
	
	\begin{proof}
		Set
		\begin{align}
			\label{eq:initialStateExtended}
			\xi_0 = (I-A^N)^{-1} \sum_{j=0}^{N-1} A^{N-1-j} B y_j\,.
		\end{align}
		This implies that under the dynamics $\xi_{k+1} = A\xi_k + By_k$ we have $\xi_N = \xi_0$. Thus the state is $N$-periodic to the $N$-periodic input $\tilde{y}$.
		Since $G(z) = \sum_{k=0}^\infty G_k z^{-k}$ for $z \in \b{T}$ with $G_0 = D$ and $G_k = CA^{k-1}B$ for $k \geq 1$, the matrix $T = (T_{k-l})_{k,l=0}^\infty$ is given by $T_l = \sum_{l=0}^\infty G_{k+lN}$ with $T_l = CA^{l-1} (I-A^N)^{-1}B$ for $l \geq 1$ and $T_0 = CA^{N-1}(I-A^N)^{-1}B + D$.
		The fact that $x_k = C\xi_k + Dy_k$ follows now by an explicit computation.
	\end{proof}

\printbibliography

\end{document}